\documentclass[11pt,leqno]{amsart}

\newtheorem{thm}{Theorem}[section]
\newtheorem{conj}[thm]{Conjecture}
\newtheorem{lemma}[thm]{Lemma}
\newtheorem{proposition}[thm]{Proposition}

\theoremstyle{definition}

\title{Centralizers of coprime automorphisms of finite groups}
\author{Cristina Acciarri}

\address{(Cristina Acciarri)  Department of Mathematics, University of Brasilia,
Brasilia-DF, 70910-900 Brazil}
\email{acciarricristina@yahoo.it}

\author{Pavel Shumyatsky} 

\address{(Pavel Shumyatsky) Department of Mathematics, University of Brasilia,
Brasilia-DF, 70910-900 Brazil}

\email{pavel@unb.br}

\keywords{finite groups, associated Lie algebras, automorphisms, centralizers}
\subjclass[2010]{Primary 20D45;  Secondary 20F40}

\begin{document}
\begin{abstract}
Let $A$ be an elementary abelian group of order $p^{k}$ with $k\geq 3$ acting on a finite $p'$-group $G$. The following results are proved.

 If $\gamma_{k-2}(C_{G}(a))$ is nilpotent of class at most $c$ for any $a\in A^{\#}$, then $\gamma_{k-2}(G)$ is nilpotent and has $\{c,k,p\}$-bounded nilpotency class.

 If, for some integer $d$ such that $2^{d}+2\leq k$, the $d$th derived group of $C_{G}(a)$ is nilpotent of class at most $c$ for any $a\in A^{\#}$, then the $d$th derived group $G^{(d)}$ is nilpotent and has $\{c,k,p\}$-bounded nilpotency class. 

Earlier this was known only in the case where $k\leq 4$.
\end{abstract}

\maketitle
%%%%%%%%%%%%%%%% SECTION INTRODUCTION %%%%%%%%%%%%%%%%%%%%%
\section{Introduction}
Let $G$ be a group admitting an action by a group $A$. We denote by $C_G(A)$ the set $\{x\in G\, \mid\, x^a=x\mbox{ for any } a\in A\}$, the centralizer of $A$ in $G$.
In this paper we deal with the case where $A$ is a noncyclic elementary abelian $p$-group and $G$ is a finite $p'$-group. Let $A^\#$ denote the set of non-identity elements of $A$. It follows from the classification of finite simple groups that if $C_G(a)$ is soluble for any $a\in A^\#$, then so is the group
$G$ (see \cite{gushu}). The case $|A|\ge p^3$ does not require the classification: the result follows from Glauberman's theorem on soluble  signalizer functors \cite{Gla}. In certain specific situations much more can be said about the structure of $G$. 

Ward showed that if $A$ has rank at least 3, and
if $C_G(a)$ is nilpotent for any $a\in A^\#$, then 
the group $G$ is nilpotent \cite{W1}. Another of
Ward's results is that if $A$ has rank at least 4, and
if $C_G(a)'$ is nilpotent for any $a\in A^\#$, then 
the derived group $G'$ is nilpotent \cite{W2}.

Later the second author of the present paper found that if under these assumptions $C_G(a)$ is nilpotent of class at most $c$ (respectively $C_G(a)'$ is nilpotent of class at most $c$) for any $a\in A^\#$, then the nilpotency class of $G$ (respectively of $G'$) is $\{c,p\}$-bounded \cite{shu1}. Throughout the article we use the term ``$\{a,b,c,\dots\}$-bounded" to mean ``bounded from above by some function depending only on the parameters $a,b,c,\dots$".

Let us denote by $\gamma_{i}(H)$ the $i$th term of the lower central series of a group $H$ and by $H^{(i)}$ the $i$th term of the derived series of $H$. In  \cite{shu1} it was  conjectured that the above results should be a part of a more general phenomenon.    
\begin{conj}
\label{conj}
Let $A$ be an elementary abelian group of order $p^{k}$ with $k\geq 3$ acting on a finite $p'$-group $G$.
\begin{enumerate}
\item[\emph{(i)}] If $\gamma_{k-2}(C_{G}(a))$ is nilpotent of class at most $c$ for any $a\in A^{\#}$, then $\gamma_{k-2}(G)$ is nilpotent and has $\{c,k,p\}$-bounded nilpotency class.
\item[\emph{(ii)}] If, for some integer $d$ such that $2^{d}+2\leq k$, the $d$th derived group of $C_{G}(a)$ is nilpotent of class at most $c$ for any $a\in A^{\#}$, then the $d$th derived group $G^{(d)}$ is nilpotent and has $\{c,k,p\}$-bounded nilpotency class. 
\end{enumerate}
\end{conj}

One indirect evidence in favor of the above conjecture is the result obtained in \cite{shu1} that the conjecture is true for Lie algebras. Yet, for long time it looked as if the Lie-theoretical result was of no help in dealing with groups. However a breakthrough has occured with the introduction in \cite{AS} of the concept of $A$-special subgroups of a group $G$. In the present paper we combine the use of the subgroups with the Lie-theoretical techniques to show, respectively in Theorems \ref{caso delta} and \ref{caso gamma}, that both parts of Conjecture \ref{conj} are true.  

In the next section we mention some standard results on the coprime action of finite groups. We also give the definition of $A$-special subgroups taken from \cite{AS} as well as remind the reader some general facts about Lie algebras associated with groups. In Section 3 these are used in the proof of Theorem \ref{caso delta}. In Section 4 we outline the proof of Theorem \ref{caso gamma}.

%%%%%%%%%%%%%%%%%%%%%% SECTION 1 %%%%%%%%%%%%%%%%%%%%%%%%%%%%%
\section{Preliminaries}
We start this section with some well-known facts about coprime automorphisms of finite groups. The next two lemmas can be found in \cite[5.3.16, 6.2.2, 6.2.4]{GO}.

\begin{lemma}
\label{FG1} 
Let $A$ be a  group of automorphisms of the finite group $G$ with $(|A|,|G|)=1$. 
\begin{enumerate}
\item If $N$ is an  $A$-invariant normal  subgroup of $G$, then \\$C_{G/N}(A)=C_G(A)N/N$;
\item If $H$ is an $A$-invariant $p$-subgroup of $G$, then $H$ is contained in an $A$-invariant Sylow $p$-subgroup of $G$.
\end{enumerate}
\end{lemma} 
 
\begin{lemma}
\label{FG2} 
Let $p$ be a prime, $G$ a finite $p'$-group acted on by an elementary abelian $p$-group $A$ of rank at least  $2$.\  Let $A_1, \dots,A_s$ be the maximal subgroups of $A$.\ If $H$ is an $A$-invariant subgroup of $G$ we have 
$H=\langle C_H(A_1),\dots,C_H(A_s)\rangle$. Furthermore if $H$ is nilpotent then $H=\prod_{i} C_{H}(A_{i})$.
\end{lemma}

In \cite{AS} we have introduced the concept of an \emph{$A$-special subgroup} of  a group $G$. The definition is as follows. 

Let $p$ be a prime and $A$ a finite elementary abelian $p$-group acting on a finite group $G$. Let $A_{1},\ldots,A_{s}$ be the subgroups of index $p$ in $A$ and $H$ a subgroup of $G$. We say that $H$ is an $A$-special subgroup of $G$ of degree $0$ if and only if $H=C_{G}(A_{j})$ for suitable $j\leq s$. Next, suppose that $i\geq 1$ and the $A$-special subgroups of $G$ of degree $i-1$ are already  defined. Then $H$ is an $A$-special subgroup of $G$ of degree $i$ if and only if there exist $A$-special subgroups $J_{1},J_{2}$ of $G$ of degree $i-1$ such that  $H=[J_{1},J_{2}]\cap C_{G}(A_{j})$ for suitable $j\leq s$.  

Here as usual $[J_{1},J_{2}]$ denotes the subgroup generated by all commutators $[x,y]$ where $x\in J_1$ and $y\in J_2$. Of course, the $A$-special subgroups of $G$ are always $A$-invariant. Assume that $A$ has order $p^k$. It is clear that for a given integer $i$ the number of $A$-special subgroups of $G$ of degree $i$ is $\{i,k,p\}$-bounded.  

Recall some properties satisfied by the $A$-special subgroups of $G$. The proofs can be found in \cite{AS}.

\begin{proposition}
\label{PAspecial}
Let $A$ be an elementary abelian $p$-group of order $p^{k}$  with $k\geq 2$ acting on  a finite $p'$-group $G$ and let $A_{1},\ldots,A_{s}$ be the maximal subgroups of $A$. Let $i\geq 0$ be an integer. 

\begin{enumerate}

\item If $i\geq 1$, then every $A$-special subgroup of $G$ of degree $i$ is contained in some $A$-special subgroup of $G$ of degree $i-1$.

\item Let $S_{i}$ be the subgroup generated by all $A$-special subgroups of $G$ of degree $i$. Then $S_{i}=G^{(i)}$.

\item If $2^{i}\leq k-1$ and $H$ is an $A$-special subgroup of $G$ of degree $i$, then $H$ is contained in the $i$th derived group of $C_{G}(B)$ for some subgroup  $B\leq A$ such that $|A/B|\leq p^{2^{i}}$.
 \end{enumerate} 
\end{proposition}

In \cite{AS} we have also established  the  following result about generation of  an $A$-invariant Sylow subgroup of $G^{(d)}$.

\begin{thm}
\label{generation1}
Let $A$ be an elementary abelian $p$-group of order $p^{k}$  with $k\geq 2$ acting on  a finite $p'$-group $G$. 
 Let  $r$ be a prime and $R$  an $A$-invariant Sylow $r$-subgroup of $G^{(d)}$, for some  integer $d\geq 0$. Let $R_{1},\ldots,R_{t}$ be the subgroups of the form $R\cap H$ where $H$ ranges through $A$-special subgroups of $G$ of degree $d$. Then $R=\langle R_{1},\dots,R_{t}\rangle$.
\end{thm} 

The proofs of the main results of the present paper are based on Lie techniques. Thus, we wish to recall here some useful Lie-theoretic machinery. Throughout the paper the term Lie algebra means Lie algebra over an associative ring with unity. Let $L$ be a Lie algebra and let $X,Y, X_{1},\ldots, X_{t}$ be subsets of $L$. We denote by $[X,Y]$ the subspace of $L$ spanned by the set $\{[x,y] \mid x\in X, y\in Y\}$  and we write $[X_{1},\ldots, X_{t}]$ for $[[X_{1},\ldots,X_{t-1}],X_{t}]$. If $t\geq 2$, we write $[X,_{\,t}Y]$ for $[[X,_{\,t-1}Y],Y]$. We denote by $\langle X\rangle$ the subalgebra of $L$ generated by $X$.

Let $G$ be a group and  let us denote by $\gamma_{i}$ the $i$th term of the lower central series of $G$. The associated Lie algebra $L(G)$ of the group $G$ is defined by  
\begin{equation*}
L(G)=\bigoplus_{i=1}^{\infty}\,\gamma_{i}/\gamma_{i+1},
\end{equation*} 
where we write additively the abelian groups $\gamma_{i}/\gamma_{i+1}$. 
Commutation in the group $G$ induces a well-defined binary operation with respect to which $L(G)$ becomes a Lie ring (Lie algebra over $\Bbb Z$). The details related to this construction can be found for example in \cite{Khu}.
If the group $G$ is nilpotent, then the Lie algebra $L(G)$ is also nilpotent and  has the same nilpotency class as $G$.

Given a subgroup $H$ of $G$, we can associate to $H$ the subalgebra  
$$
L(G,H)=\bigoplus_{i=1}^{\infty}\,(H\cap \gamma_{i})\gamma_{i+1}/\gamma_{i+1}.
$$

If a group $A$ acts on $G$, then $A$ acts naturally also on each quotient $\gamma_{i}/\gamma_{i+1}$ and this action extends uniquely to an action by automorphisms on the whole Lie algebra $L(G)$. Lemma \ref{FG1}(1) shows that if $(|A|,|G|)=1$, then 
$$C_{L(G)}(A)=\bigoplus_{i}\,C_{\gamma_{i}}(A)\gamma_{i+1}/\gamma_{i+1}.$$ Therefore in the case where $(|A|,|G|)=1$ we have $C_{L(G)}(A)=L(G,C_G(A))$.

Later on we will require the following lemma.
\begin{lemma}\label{span} Let $L$ be a Lie algebra such that $pL=L$ where $p$ is a prime, and let $A$ be a finite elementary abelian $p$-group acting by automorphisms on $L$. Let   $A_{1},\ldots,A_{s}$ be  the maximal subgroups of $A$.  Suppose that $L$ is generated by $A$-invariant subspaces $R_{1},\ldots,R_{t}$ with the  property that  for any integers $i,j$ and $k$ there exists some integer $m$ such that 
$$
[R_{i},R_{j}]\cap C_{L}(A_{k})\leq R_{m}.
$$ 
Then $L$ is spanned by $R_{1},\ldots,R_{t}$.
\end{lemma}
\begin{proof} Clearly, $L$ is a linear span of subspaces of the form $[R_{i_{1}},\ldots,R_{i_{w}}]$, where $R_{i_{1}},R_{i_{2}},\ldots,R_{i_{w}}$ are not necessarily distinct elements of $\{R_{1}.\ldots,R_{t}\}$. So choose $R_{i_{1}},R_{i_{2}},\ldots,R_{i_{w}}\in\{R_{1}.\ldots,R_{t}\}$ and put $R=[R_{i_{1}},\ldots,R_{i_{w}}]$. It is sufficient to show that $R$ is contained in $\sum_{j} R_{j}$. We argue by  induction on $w$. If $w=1$, then $R=R_{j}$, for some $j$ and there is nothing to prove.  Assume that $w\geq 2$ and put $R_{0}=[R_{i_{1}},\ldots,R_{i_{w-1}}]$. Thus $R=[R_{0},R_{i_{w}}]$. Since $R$ is an $A$-invariant subspace  it follows from Lemma \ref{FG2} that  $R=\sum_{\lambda\leq s} C_{R}(A_{\lambda})$. By the inductive hypothesis  $R_{0}\leq \sum_{j} R_{j}$. Therefore we have  
 \begin{equation*}
 \begin{split}
 C_{R}(A_{\lambda})=[R_{0},R_{i_{w}}]\cap C_{L}(A_{\lambda})&\leq [\sum_{j} R_{j},R_{i_{w}}]\cap C_{L}(A_{\lambda})\\
 &\leq \sum_{j}\big([R_{j},R_{i_{w}}]\cap C_{L}(A_{\lambda})\big).
 \end{split}
 \end{equation*}
By the hypothesis  each summand $[R_{j},R_{i_{w}}]\cap C_{L}(A_{\lambda})$ is contained in  $R_{m}$, for some integer $m$, and so it follows that $R\leq \sum_{j}R_{j}$, as desired. 
\end{proof}

%%%%%%%%%%%%%%%%%%%%% SECTION 2 %%%%%%%%%%%%%%%%%%%%%%%%%%%%%%
\section{Proof of the Main Result}
The aim of this section is to prove  part (ii) of Conjecture \ref{conj}. 

\begin{thm}
\label{caso delta}
Let $c$ be a positive integer, $p$ a prime,  and  $A$  an elementary abelian group of order $p^{k}$ with $k\geq 3$ acting on a finite $p'$-group $G$.  If, for some integer $d$ such that $2^{d}+2\leq k$, the $d$th derived group of $C_{G}(a)$ is nilpotent of class at most $c$ for any $a\in A^{\#}$, then the $d$th derived group $G^{(d)}$ is nilpotent and has $\{c,k,p\}$-bounded nilpotency class. 
\end{thm}

First we wish to show that under the hypotheses of the above theorem the $d$th derived group $G^{(d)}$ is nilpotent. In what follows  we write $F(K)$ for the Fitting subgroup of a group $K$ and $O_{\pi}(K)$ for the maximal normal $\pi$-subgroup of $K$, where $\pi$ is a set of primes.

\begin{lemma}
\label{nilpotency for G^d}
Assume the hypotheses of Theorem \ref{caso delta}. Then $G^{(d)}$ is nilpotent. 
\end{lemma}

\begin{proof} Suppose that the lemma is false and let $G$ be a counterexample of minimal order. Since the $d$th derived group of $C_{G}(a)$ is nilpotent, it follows that $C_{G}(a)$ is soluble for any $a\in A^{\#}$.  Therefore Glauberman's result  on soluble signalizer functors \cite{Gla}  implies that $G$ is soluble. Assume that $G$ has two  distinct minimal  $A$-invariant normal subgroups $M_{1}$ and $M_{2}$. By minimality  the image of $G^{(d)}$ in $G/M_{1}$ and in $G/M_{2}$ is nilpotent. Thus  the image of $G^{(d)}$ must be nilpotent in the quotient $G/(M_{1}\cap M_{2})$. This is a contradiction since $M_{1}\cap M_{2}=1$. 

Therefore $G$ has a unique minimal $A$-invariant normal subgroup $M$. Again the quotient $G^{(d)}/M$ is nilpotent. It is clear that $M$ is an elementary abelian $q$-group for some prime $q$. Let $A_{1},\ldots,A_{s}$ be the maximal subgroups of $A$. By Lemma \ref{FG2} $M=M_{1}M_{2}\cdots M_{s}$, where $M_{i}=C_{M}(A_{i})$ for $i\leq s$. Since $G^{(d)}$ is not nilpotent, it is not a $q$-group. Therefore by Lemma \ref{FG1}(2) $G^{(d)}$ contains  an $A$-invariant Sylow  $r$-subgroup  $R$ for some prime $r\neq q$. Theorem \ref{generation1} tells us that $R$ is generated by its intersections with $A$-special subgroups of degree $d$. Thus, $R=\langle R_{1},\ldots,R_{t}\rangle$, where $R_{j}=R\cap H_{j}$ for some $A$-special subgroup $H_{j}$ of $G$ of degree $d$. Now  fix  the integers $i$ and $j$ and consider the subgroup $\langle M_{i},R_{j}\rangle$. Since $2^{d}\leq k-1$ it follows from Proposition \ref{PAspecial}(3) that $H_{j}$ is contained in $C_{G}(B)^{(d)}$ for some subgroup $B$ of $A$ such that $|A/B|\leq p^{2^{d}}$. On the other hand $M_{i}\leq C_{G}(A_{i})$ and note that the intersection $B\cap A_i$ is not trivial. Therefore there exists $a\in A^{\#}$ such that $M_{i}\leq C_{G}(a)$ and $H_{j}\leq C_{G}(a)^{(d)}$. It follows that $H_{j}$ is contained in $F(C_{G}(a))$. Since $M_{i}$ is contained in a normal abelian subgroup of $G$ and also in $C_{G}(a)$, it follows that $\langle M_{i},R_{j} \rangle$ is nilpotent. Bearing in mind that $M$ is a $q$-group and $R$ is an $r$-group we deduce that $[M_{i},R_{j}]=1$ and this holds for any $i,j$. Recall that $M=M_{1}M_{2}\cdots M_{s}$ and $R=\langle R_{1},\ldots,R_{t}\rangle$. Therefore $[M,R]=1$.    

The fact that $G^{(d)}/M$ is nilpotent implies that also $C_{G}(M)\cap G^{(d)}$ is nilpotent. Hence, every $q'$-element of $C_{G}(M)\cap G^{(d)}$ belongs to $O_{q'}(G)$. On the other hand $O_{q'}(G)$ is trivial since $M$ is the unique minimal $A$-invariant normal subgroup of $G$. Thus, we obtain a contradiction as we have just shown that $R$ centralizes $M$. 
\end{proof}

\begin{proof}[Proof of Theorem  \ref{caso delta}]
By Lemma \ref{nilpotency for G^d} $G^{(d)}$ is nilpotent.   Let $L=L(G^{(d)})$ be the Lie algebra associated with $G^{(d)}$. Then $pL=L$ and $L$ has the same nilpotency class as $G^{(d)}$. The group $A$ naturally acts by automorphisms on the Lie algebra $L$. From the hypothesis that $C_{G}(a)^{(d)}$ is nilpotent of  class at most $c$ we obtain that $C_{L}(a)^{(d)}$ is nilpotent of class at most $c$ for any $a\in A^{\#}$.  

Let $K=L\otimes \mathbb{Z}[\omega]$, where $\omega$ is a primitive $p$th root of unity. Then for each $i\geq 0$ and $a\in A^{\#}$ we have
$$C_{K}(a)^{(i)}=C_{L}(a)^{(i)} \otimes \mathbb{Z}[\omega].$$ Hence, the nilpotency of $C_{L}(a)^{(d)}$  implies that also  $C_{K}(a)^{(d)}$ is nilpotent of class at most $c$ for any $a\in A^{\#}$.

We are in the position to apply Theorem 2.7(2) from \cite{shu1} and  conclude that  $K^{(d)}$ is nilpotent of $\{c,k,p\}$-bounded class. The same holds for $L^{(d)}$. Let us denote the nilpotency class of $L^{(d)}$ by $e$.

By Proposition \ref{PAspecial}(2) $G^{(d)}=\langle H_{1},H_{2},\ldots,H_t\rangle$, where $H_{i}$ are the $A$-special subgroups of $G$ of degree $d$. Since $2^{d}+2\leq k$, Proposition \ref{PAspecial}(3) tells us that each $A$-special subgroup  $H_{i}$  of $G$ of degree $d$ is contained in $C_{G}(B)^{(d)}$, for some  subgroup $B$ of $A$ such that $|A/B|\leq p^{2^{d}}$. Let $A_{1},\ldots,A_{s}$ be the maximal subgroups of $A$. For any $A_{j}$ the intersection $B\cap A_{j}$ is not trivial. Thus, there exists $a\in A^{\#}$ such that the centralizer $C_{G}(A_{j})$ is contained in $C_{G}(a)$ and  $H_{i}$ is contained in  $C_{G}(a)^{(d)}$. Since $C_{G}(a)^{(d)}$ is nilpotent of class at most $c$  we deduce that 
\begin{equation}
\label{relationdelta}
[C_{G}(A_{j}),_{\,c+1}H_{i}]=1.
\end{equation}

Next we define recursively  what will be called $A$-subalgebras of $L$. For each $A$-special subgroup $H_{i}$ of $G$ of degree $d$ we consider the corresponding  subalgebra $L(G^{(d)},H_{i})$ of $L$ and we define the  \textit{$A$-subalgebras} as follows: 
 
A subalgebra $R$ is an $A$-subalgebra of level $0$ if and only if $R=L(G^{(d)},H_j)$ for suitable $j\leq t$. Next, suppose that $l\geq 1$ and the $A$-subalgebras  of level $l-1$ are defined. Then $R$ is an $A$-subalgebra of level $l$ if and only if there exist $A$-subalgebras $R_1,R_2$ of level $l-1$ such that $R=[R_1,R_2]\cap C_{L}(A_j)$ for suitable $j\leq s$.  

It is clear that every $A$-subalgebra is $A$-invariant and is contained in $C_L(A_j)$ for some $j\leq s$. Since $G^{(d)}=\langle H_1,H_2,\ldots,H_{t}\rangle$ it follows that $L$ is generated by the $A$-subalgebras of level $0$. It is easy to check that if $R$ is an $A$-subalgebra of level $l$, then $G$ contains an $A$-special subgroup $H$ of degree $d+l$ such that $R\leq L(G^{(d)},H)$.
 
It follows from the definition and Proposition \ref{PAspecial}(1) that for any $A$-special subgroups $J_1$ and $J_2$ and for every $j\leq s$ there exists an $A$-special subgroup $J_3$ such that 
\begin{equation}
\label{propgroup}
[J_1,J_2]\cap C_{G}(A_{j})\leq J_3.
\end{equation}

From this we deduce the corresponding properties of $A$-subalgebras.   
  
 \begin{itemize}
 \item[(P1)] If $l\geq 1$, then every $A$-subalgebra of level $l$ is contained in some $A$-subalgebra of level $l-1$.
 
\item[(P2)] If $j\leq s$, then for any $A$-subalgebras $R_1,R_2$ of level $l$ there exists an $A$-subalgebra $R_3$ of the same level $l$ such that 
 $$
 [R_1,R_2]\cap C_{L}(A_{j})\leq R_3.
 $$  
  \end{itemize}
 
In the group $G$ we have the relation (\ref{relationdelta}). Therefore in the Lie algebra we have $[C_{L}(A_{j}),_{\,c+1}L(G^{(d)},H_{i})]=0$. Taking into account that every $A$-subalgebra is contained in some $L(G^{(d)},H_{i})$ and that $L=\sum_{j}C_{L}(A_{j})$ we deduce $[L,_{\,c+1}L(G^{(d)},H_{i})]=0$, and, in particular, 
 \begin{equation}
 \label{equ5}
[L,_{\,c+1}R]=0
 \end{equation} 
for every $A$-subalgebra $R$.

Now we wish to show that for any $l\geq 0$ the $l$th derived algebra $L^{(l)}$ is spanned by the $A$-subalgebras of level $l$. The property (P2) and Lemma \ref{span} show that this happens if and only if $L^{(l)}$ is generated by the $A$-subalgebras of level $l$. Since $L$ is generated by the $A$-subalgebras of level $0$, this is obvious for $l=0$. Now assume that $l\geq 1$ and use induction on $l$. The inductive hypothesis will be  that $L^{(l-1)}$ is spanned by the $A$-subalgebras of level $l-1$. Let $N$ be the subalgebra of $L$ generated by the $A$-subalgebras of level $l$. We already know that in fact $N$ is spanned by the $A$-subalgebras of level $l$. Let us show that actually $N$ is an ideal in $L^{(l-1)}$. Choose an $A$-subalgebra $R_1$ of level $l$ and an $A$-subalgebra $R_2$ of level $l-1$. The properties (P1) and (P2) show that $[R_1,R_2]\leq N$. Since $L^{(d-1)}$ is spanned by the $A$-subalgebras of level $l-1$, we conclude that indeed $N$ is an ideal in $L^{(l-1)}$. Note that $L^{(l-1)}/N$ is abelian. Thus  $L^{(l)}=[L^{(l-1)},L^{(l-1)}]\leq N$. On the other hand, by construction it is clear that  $N$ is contained in $L^{(l)}$. Hence, $N=L^{(l)}$.

We will now  prove that $L$ is nilpotent of  $\{c,k,p\}$-bounded class. 
Let $Z=Z(L^{(d)})$. Then $[Z,X,Y]=[Z,Y,X]$ for any subsets $X,Y$ of $L^{(d-1)}$.  Set $r=cn+1$, where $n$ is the number of $A$-subalgebras of level $d-1$ and note that $n$ is a $\{k,p\}$-bounded number. Since  $L^{(d-1)}$  is spanned by the $A$-subalgebras of level $d-1$, i.e., $L^{(d-1)}=\sum_{i\leq n} R_{i}$,  we  can write 
\begin{equation}
\label{equ6delta}
[Z,_{\,r}L^{(d-1)}]=\sum\,[Z,_{\,u_{1}}\,R_1,\ldots,_{\,u_{n}}R_{n}], 
\end{equation}
where $u_{1}+\cdots+u_{n}=r$ and $R_1,\ldots,R_{n}$ are the $A$-subalgebras of level $d-1$. The number $r$ is big enough to ensure that $u_{j}\geq c+1$ for some $j\leq n$. It follows from (\ref{equ5}) that  each summand in (\ref{equ6delta}) is equal to zero. Thus $[Z,_{r}L^{(d-1)}]=0$ and $Z\leq Z_{r}(L^{(d-1)})$, where $Z_{r}(L^{(d-1)})$ is the $r$th term of the upper central series of $L^{(d-1)}$. Now repeating this argument for $L^{(d-1)}/Z$, $L^{(d-1)}/Z_{2}(L^{(d)})$ and so on, we conclude that $L^{(d)}\leq Z_{er}(L^{(d-1)})$  and therefore $L^{(d-1)}$ is nilpotent of class at most $er+1$. After that we repeat the arguments for $L^{(d-2)}/Z(L^{(d-1)})$, $L^{(d-2)}/Z_{2}(L^{(d-1)})$ etc. After boundedly many repetitions we conclude that $L$ is nilpotent of  $\{c,k,p\}$-bounded class. 

Finally we remark that since the nilpotency class of $G^{(d)}$ equals that of $L$, the result follows. The proof is complete.
 \end{proof}
 
%%%%%%%%%%%%%%%%%% SECTION 3 %%%%%%%%%%%%%%%%%%%%%%%%%%%%%%%%%% 
 
 \section{The other part of the conjecture}
In this section we will outline a proof of the following result.
 \begin{thm}
\label{caso gamma}
Let $A$ be an elementary abelian group of order $p^{k}$ with $k\geq 3$ acting on a finite $p'$-group $G$. If $\gamma_{k-2}(C_{G}(a))$ is nilpotent of class at most $c$ for any $a\in A^{\#}$, then $\gamma_{k-2}(G)$ is nilpotent and has $\{c,k,p\}$-bounded nilpotency  class.
\end{thm}
The proof  of Theorem \ref{caso gamma} is very similar to that of Theorem \ref{caso delta}. Very often the changes that need to be done are quite obvious and therefore we omit many details. Most essential difference as compared with Theorem \ref{caso delta} is that the role of $A$-special subgroups will now be played by \textit{$\gamma$-$A$-special subgroups} of $G$. These were introduced in \cite{AS}. Let us recall the definition. 
 
Let $p$ be a prime and $A$ a finite elementary abelian $p$-group acting on a finite 
group $G$. Let $A_{1},\ldots,A_{s}$ be the subgroups of index $p$ in $A$ and $H$ a subgroup of $G$. We say that $H$ is a $\gamma$-$A$-special subgroup of $G$ of degree $1$ if and only if $H=C_{G}(A_{j})$ for suitable $j\leq s$. Next, suppose that $i\geq 2$ and the $\gamma$-$A$-special subgroups of $G$ of degree $i-1$ are already defined. Then $H$ is a $\gamma$-$A$-special subgroup of $G$ of degree $i$ if and only if there exists a $\gamma$-$A$-special subgroup $J$ of $G$ of degree $i-1$ such that  $H=[J,C_{G}(A_{j})]\cap C_{G}(A_{n})$ for suitable $j,n\leq s$. Note that for a given integer $i$ the number of $\gamma$-$A$-special subgroups of $G$ of degree $i$ is $\{i,k,p\}$-bounded. The following properties of $\gamma$-$A$-special subgroups have been established in \cite{AS}.

\begin{proposition}
\label{gammaPAspecial}
Let $A$ be an elementary abelian $p$-group of order $p^{k}$ with $k\geq 2$ acting on a finite $p'$-group $G$
and $A_{1},\ldots,A_{s}$ the maximal subgroups of $A$. Let $i\geq 1$ be an integer. 

\begin{enumerate}
\item If $i\geq 2$, then every $\gamma$-$A$-special subgroup of $G$ of degree $i$ is contained in some $\gamma$-$A$-special subgroup of $G$ of degree $i-1$.

\item Let $S_{i}$ be the subgroup generated by all $\gamma$-$A$-special subgroups of $G$ of degree $i$. Then $S_{i}=\gamma_{i}{(G)}$.

\item If $i\leq k-1$ and $H$ is a $\gamma$-$A$-special subgroup of $G$ of degree $i$, then $H\leq \gamma_{i}(C_{G}(B))$ for some subgroup  $B\leq A$ such that $|A/B|\leq p^{i}$.

 \end{enumerate}  
  
\end{proposition}

We will also require the following analogue of Lemma \ref{span}.
\begin{lemma}
\label{gammaspan}
Let $L$ be a Lie algebra such that $pL=L$ where $p$ is a prime, and let $A$ be a finite elementary abelian $p$-group acting by automorphisms on $L$. Let   $A_{1},\ldots,A_{s}$ be  the maximal subgroups of $A$.  Suppose that $L$ is generated by $A$-invariant subspaces $R_{1},\ldots,R_{t}$ with the  property that  for any integers $i,j$ and $k$ there exists some integer $m$ such that 
$$
[R_{i},C_{L}(A_{j})]\cap C_{L}(A_{k})\leq R_{m}.
$$ 
Then $L$ is spanned by $R_{1},\ldots,R_{t}$.
\end{lemma}
 
Now we sketch out the proof of Theorem \ref{caso gamma}.

\begin{proof}[Proof of Theorem \ref{caso gamma}] First, we notice that $\gamma_{k-2}(G)$ is nilpotent. The proof of the nilpotency of $\gamma_{k-2}(G)$ is  similar to that of Lemma \ref{nilpotency for G^d}. Next, we let $L=L(\gamma_{k-2}(G))$ be the Lie algebra associated with $\gamma_{k-2}(G)$. Then $pL=L$ and $L$ has the same nilpotency class as $\gamma_{k-2}(G)$. The group $A$ naturally acts by automorphisms on $L$ and, since $\gamma_{k-2}(C_{G}(a))$ is nilpotent of class at most $c$, it follows that $\gamma_{k-2}(C_{L}(a))$ is nilpotent of class at most $c$ for any $a\in A^{\#}$. Put  $K=L\otimes \mathbb{Z}[\omega]$, where $\omega$ is a primitive $p$th root of unity.  The nilpotency of $\gamma_{k-2}(C_{L}(a))$ implies that also  $\gamma_{k-2}(C_{K}(a))$ is nilpotent of class at most $c$ for any $a\in A^{\#}$. Theorem 2.7(1) of \cite{shu1} now tells us that $\gamma_{k-2}(K)$ is nilpotent of $\{c,k,p\}$-bounded class. Hence, also the nilpotency class of $\gamma_{k-2}(L)$ is $\{c,k,p\}$-bounded. We denote the nilpotency class of $\gamma_{k-2}(L)$ by  $e$. 

Let $H_1,H_2,\ldots,H_t$ be the $\gamma$-$A$-special subgroups of $G$ of degree $k-2$. By Proposition \ref{gammaPAspecial}(2) $\gamma_{k-2}(G)=\langle H_{1},H_{2},\ldots,H_t\rangle$. Since $k-2\leq k-1$, Proposition \ref{gammaPAspecial}(3) tells us that each subgroup  $H_{i}$ is contained in $\gamma_{k-2}(C_{G}(B))$ for some  subgroup $B$ of $A$ such that $|A/B|\leq p^{k-2}$. Let $A_{1},\ldots,A_{s}$ be the maximal subgroups of $A$.  For any $A_{j}$ the intersection $B\cap A_{j}$ is not trivial. Thus, there exists $a\in A^{\#}$ such that the centralizer $C_{G}(A_{j})$ is contained in $C_{G}(a)$ and  $H_{i}$ is contained in  $\gamma_{k-2}(C_{G}(a))$. Since $\gamma_{k-2}(C_{G}(a))$ is nilpotent of class at most $c$, we have 
\begin{equation}
\label{relation}
[C_{G}(A_{j}),_{\,c+1}H_{i}]=1.
\end{equation}

Next, we define recursively what will be called  $\gamma$-$A$-subalgebras of $L$. The definition is similar to that of $A$-subalgebras used in the previous section.

For each $\gamma$-$A$-special subgroup $H_{i}$ of $G$ of degree $k-2$ we consider the corresponding subalgebra $L(\gamma_{k-2}(G),H_{i})$. A subalgebra $R$ is a $\gamma$-$A$-subalgebra of level $1$ if and only if there exists $j\leq t$ such that $R=L(\gamma_{k-2}(G),H_j)$. Further, suppose that $l\geq 2$ and the $\gamma$-$A$-subalgebras of level $l-1$ are already defined. Then $R$ is a $\gamma$-$A$-subalgebra of level $l$ if and only if there exists a $\gamma$-$A$-subalgebra $R_1$ of level $l-1$ such that $R=[R_1,C_{L}(A_{j})]\cap C_{L}(A_m)$ for suitable $j,m\leq s$. 

Since (\ref{relation}) holds in the group $G$, we deduce that $[L,_{\,c+1}R]=0$ for every $\gamma$-$A$-subalgebra $R$. Furthermore, using  Lemma \ref{gammaspan} one can show that for every $l\geq 1$ the $l$th term  $\gamma_{l}(L)$ of the lower central series of $L$ is spanned by the $\gamma$-$A$-subalgebras of level $l$.

Finally, we use the above remarks to prove that $L$ is  nilpotent of  $\{c,k,p\}$-bounded class. This part of the proof is pretty much the same as that of Theorem \ref{caso delta}. Since $\gamma_{k-2}(G)$ has the same nilpotency class as $L$, the theorem follows.
\end{proof}

%%%%%%%%%%%%%%%%% SECTION THANKS TO %%%%%%%%%%%%%%%%%%%%%%%%%%%%%%

\section{Acknowledgments} 
This research was supported by CNPq-Brazil.

%%%%%%%%%%%%%%%%%%% BIBLIOGRAPHY %%%%%%%%%%%%%%%%%%%%%%%%%%%%

\end{document}